\documentclass[11pt]{article}
\usepackage{amsmath,amssymb,amsthm,color}
\usepackage{comment}
\usepackage{color}

\newtheorem{theorem}{Theorem}[section]
\newtheorem{corollary}{Corollary}[section]
\newtheorem{lemma}{Lemma}[section]
\newtheorem{proposition}{Proposition}[section]
\newtheorem{definition}{Definition}[section]

\newtheorem{remark}{Remark}[section]

\def\R{{\mathfrak R}\, }
\def\D{{\mathfrak D}\, }

\begin{document}
\begin{center}{\bf \LARGE Characterization of Lie multiplicative derivation on alternative rings}\\
\vspace{.2in}
\noindent {\bf Bruno Leonardo Macedo Ferreira}\\
{\it Technological Federal University of Paran\'{a},\\
Professora Laura Pacheco Bastos Avenue, 800,\\
85053-510, Guarapuava, Brazil.}\\
e-mail: brunoferreira@utfpr.edu.br
\\
and
\\
\noindent {\bf Henrique Guzzo Jr.}\\
{\it Institute of Mathematics and Statistics\\
University of S\~{a}o Paulo,\\
Mat\~{a}o Street, 1010,\\   
05508-090, S\~{a}o Paulo, Brazil.}\\
e-mail: guzzo@ime.usp.br

\end{center}
\begin{abstract}
In this paper we generalize the result valid for associative rings due \cite[Martindale III]{Mart} and \cite[Bre$\check{s}$ar]{bresar} to alternative rings.
Let $\R$ be an unital alternative ring, and $\D \colon \R \rightarrow \R$ is a Lie multiplicative derivation. Then $\D$ is the form $\delta + \tau$ where $\delta$ is an additive derivation of $\R$ and $\tau$ is a map from $\R$ into its center $\mathcal{\R}$, which maps commutators into the zero.
\end{abstract}
{\bf {\it Running Head:}} Lie multiplicative derivation
\section{Alternative rings and Lie multiplicative derivation}

Let $\mathfrak{R}$ be a unital ring not necessarily associative or commutative and consider the following convention for its multiplication operation: $xy\cdot z = (xy)z$ and $x\cdot yz = x(yz)$ for $x,y,z\in \mathfrak{R}$, to reduce the number of parentheses. We denote the {\it associator} of $\mathfrak{R}$ by $(x,y,z)=xy\cdot z-x\cdot yz$ for $x,y,z\in \mathfrak{R}$. And $[x,y] = xy - yx$ is the usual Lie product of $x$ and $y$, with $x,y \in \mathfrak{R}$.

Let $\mathfrak{R}$ be a ring and $\D\colon \mathfrak{R}\rightarrow \mathfrak{R}$ a mapping of $\mathfrak{R}$ into itself. We call $\D$ a {\it Lie multiplicative derivation} of $\mathfrak{R}$ into itself if for all $x, y \in \mathfrak{R}$ 
\begin{eqnarray*}
\D\big([x,y]) = [\D(x),y] + [x,\D(y)].
\end{eqnarray*}
A ring $\mathfrak{R}$ is said to be {\it alternative} if $(x,x,y)=0=(y,x,x)$ for all $x,y\in \mathfrak{R}$. One easily sees that any associative ring is an alternative ring. An alternative ring $\mathfrak{R}$ is called {\it k-torsion free} if $k\,x=0$ implies $x=0,$ for any $x\in \mathfrak{R},$ where $k\in{\mathbb Z},\, k>0$, and {\it prime} if $\mathfrak{AB} \neq 0$ for any two nonzero ideals $\mathfrak{A},\mathfrak{B}\subseteq \mathfrak{R}$.
The {\it nucleus} of an alternative ring $\mathfrak{R}$ \mbox{is defined by} $$\mathcal{N}(\mathfrak{R})=\{r\in \mathfrak{R}\mid (x,y,r)=0=(x,r,y)=(r,x,y) \hbox{ for all }x,y\in \mathfrak{R}\}.$$
And the {\it centre} of an alternative ring $\mathfrak{R}$ \mbox{is defined by} $$\mathcal{Z}(\mathfrak{R})=\{r\in \mathcal{N}\mid [r, x] = 0 \hbox{ for all }x \in \mathfrak{R}\}.$$
\begin{theorem}\label{meu}
Let $\mathfrak{R}$ be a $3$-torsion free alternative ring. So
$\mathfrak{R}$ is a prime ring if and only if $a\mathfrak{R} \cdot b=0$ (or $a \cdot \mathfrak{R}b=0$) implies $a = 0$ or $b =0$ for $a, b \in \mathfrak{R}$. 
\end{theorem}
\begin{proof}
Clearly all alternative rings satisfying the properties $a\mathfrak{R} \cdot b=0$ (or $a \cdot \mathfrak{R}b=0$) are prime rings.
Suppose $\mathfrak{R}$ is a prime ring by \cite[Lemma $2.4$, Theorem $A$ and Proposition $3.5$]{slater} we have $\mathfrak{R} = \mathcal{A}_{0} \supseteq \mathcal{A}_{1} \supseteq \cdots \supseteq \mathcal{A}_{n} = \mathcal{A} \neq (0)$ is a chain of subrings of $\mathfrak{R}$. If $a\mathfrak{R} \cdot b=0$ (or $a \cdot \mathfrak{R}b=0$) hence $a\mathcal{A} \cdot b=0$ (or $a \cdot \mathcal{A}b=0$) follows \cite[Proposition 3.5 (e)]{slater} that $a= 0$ or $b=0$. 
 \end{proof}

\begin{definition}
A ring $\R$ is said to be flexible if satisfies
$$(x,y,x) = 0 \ \ for \ all \ x,y \in \R.$$
It is known that alternative rings are flexible.
\end{definition}

\begin{proposition}\label{Lflexivel}
Let $\R$ be a alternative ring then $\R$ satisfies 
$$(x,y,z) + (z,y,x) = 0 \ \ for \ \ all \ x,y,z \in \R.$$
\end{proposition}
\begin{proof} Just linearize the identity $(x,y,x) = 0$.
\end{proof}
A nonzero element $e_{1}\in \mathfrak{R}$ is called an {\it idempotent} if $e_{1}e_{1}=e_{1}$ and a {\it nontrivial idempotent} if it is an idempotent different from the multiplicative identity element of $\mathfrak{R}$. Let us consider $\mathfrak{R}$ an alternative ring and fix a nontrivial idempotent $e_{1}\in\mathfrak{R}$. Let \mbox{$e_2 \colon\mathfrak{R}\rightarrow\mathfrak{R}$} and $e'_2 \colon\mathfrak{R}\rightarrow\mathfrak{R}$ be linear operators given by $e_2(a)=a-e_1a$ and $e_2'(a)=a-ae_1.$ Clearly $e_2^2=e_2,$ $(e_2')^2=e_2'$ and we note that if $\mathfrak{R}$ has a unity, then we can consider $e_2=1-e_1\in \mathfrak{R}$. Let us denote $e_2(a)$ by $e_2a$ and $e_2'(a)$ by $ae_2$. It is easy to see that $e_ia\cdot e_j=e_i\cdot ae_j~(i,j=1,2)$ for all $a\in \mathfrak{R}$. Then $\mathfrak{R}$ has a Peirce decomposition
$\mathfrak{R}=\mathfrak{R}_{11}\oplus \mathfrak{R}_{12}\oplus
\mathfrak{R}_{21}\oplus \mathfrak{R}_{22},$ where
$\mathfrak{R}_{ij}=e_{i}\mathfrak{R}e_{j}$ $(i,j=1,2)$ \cite{He}, satisfying the following multiplicative relations:
\begin{enumerate}\label{asquatro}
\item [\it (i)] $\mathfrak{R}_{ij}\mathfrak{R}_{jl}\subseteq\mathfrak{R}_{il}\
(i,j,l=1,2);$
\item [\it (ii)] $\mathfrak{R}_{ij}\mathfrak{R}_{ij}\subseteq \mathfrak{R}_{ji}\
(i,j=1,2);$
\item [\it (iii)] $\mathfrak{R}_{ij}\mathfrak{R}_{kl}=0,$ if $j\neq k$ and
$(i,j)\neq (k,l),\ (i,j,k,l=1,2);$
\item [\it (iv)] $x_{ij}^{2}=0,$ for all $x_{ij}\in \mathfrak{R}_{ij}\ (i,j=1,2;~i\neq j).$
\end{enumerate}
In this paper we consider that $\R$ is $2,3$-torsion free alternative ring with satisfying:
\begin{enumerate}\label{identi}
\item [\it (1)] If $x_{ij}\R_{ji} = 0$ then $x_{ij} = 0$ ($i \neq j$);
\item [\it (2)] If $x_{11}\R_{12} = 0$ or $\R_{21}x_{11} = 0$ then $x_{11} = 0$;
\item [\it (3)] If $\R_{12}x_{22} = 0$ or $x_{22}\R_{21} = 0$ then $x_{22} = 0$;
\item [\it (4)] If $z \in \mathcal{Z}$ with $z \neq 0$ then $z\R = \R$.
\end{enumerate}

\begin{remark}
Note that prime alternative rings satisfy (\ref{identi}), (2), (3). In deed, we firstly show (\ref{identi}).

Let be $x_{ij}\R_{ji} = 0$ so $x_{ij}(\R e_i) = 0$. As $\R$ is $3$-torsion free alternative ring and $e_i$ is a nontrivial idempotent we have $x_{ij} = 0$, by Theorem \ref{meu}.

(2) If $x_{11}\R_{12} = 0$ or $\R_{21}x_{11} = 0$ so $x_{11}(\R e_2) = 0$ or $(e_2 \R) x_{11} = 0$. Thus $x_{11} = 0$ because $e_2$ is a nontrivial idempotent.

(3) It is similar to (2).  
\end{remark}

\begin{proposition}\label{prop2}
Let $\R$ be a $2,3$-torsion free alternative ring satisfying $(1)$, $(2)$, $(3)$.
\begin{enumerate}
\item [$(\spadesuit)$] If $[a_{11}+ a_{22}, \mathfrak{R}_{12}] = 0$, then $a_{11} + a_{22} \in \mathcal{Z}(\mathfrak{R})$,
\item [$(\clubsuit)$] If $[a_{11}+ a_{22}, \mathfrak{R}_{21}] = 0$, then $a_{11} + a_{22} \in \mathcal{Z}(\mathfrak{R})$.
\end{enumerate}
\end{proposition}
\begin{proof}
We will prove only $(\spadesuit)$ because $(\clubsuit)$ it is similar to $(\spadesuit)$.
For any $x_{11} \in \R_{11}$ and $y_{12} \in \R_{12}$, by Proposition \ref{Lflexivel} we have
\begin{eqnarray*}
(a_{11}x_{11})y_{12} &=& a_{11}(x_{11}y_{12}) =(x_{11}y_{12})a_{22} =x_{11}(y_{12}a_{22}) =x_{11}(a_{11}y_{12}) \\&=& (x_{11}a_{11})y_{12}.
\end{eqnarray*}
It follows from $(2)$ that $a_{11}x_{11} = x_{11}a_{11}$.
Now for any $x_{12} \in \R_{12}$ and $y_{22} \in \R_{22}$, by Proposition \ref{Lflexivel}
\begin{eqnarray*}
x_{12}(y_{22}a_{22}) &=& (x_{12}y_{22})a_{22} = a_{11}(x_{12}y_{22}) = (a_{11}x_{12})y_{22}= (x_{12}a_{22})y_{22} \\&=& x_{12}(a_{22}y_{22}).
\end{eqnarray*}
By $(3)$, we see that $a_{22}y_{22} = y_{22}a_{22}.$
Let $x_{21} \in \R_{21}$ and $y_{12} \in \R_{12}$ be arbitrary. Applying identity above and Proposition \ref{Lflexivel}, we get
\begin{eqnarray*}
(a_{22}x_{21})y_{12} &=& a_{22}(x_{21}y_{12}) = (x_{21}y_{12})a_{22} = x_{21}(y_{12}a_{22}) = x_{21}(a_{11}y_{12}) \\&=& (x_{21}a_{11})y_{12},
\end{eqnarray*}
which, by $(1)$, implies that $a_{22}x_{21} = x_{21}a_{11}.$
Now, for any $x \in \R$, using identities above we get
\begin{eqnarray*}
(a_{11} + a_{22})x &=& (a_{11} + a_{22})(x_{11} + x_{12} + x_{21} + x_{22})
\\&=& a_{11}x_{11} + a_{11}x_{12} + a_{22}x_{21} + a_{22}x_{22}
\\&=& x_{11}a_{11} + x_{12}a_{22} + x_{21}a_{11} + x_{22}a_{22}
\\&=& (x_{11} + x_{12} + x_{21} + x_{22})(a_{11} + a_{22})
\\&=& x(a_{11} + a_{22}).
\end{eqnarray*}
Thus $a_{11} + a_{22} \in \mathcal{Z}(\R)$.
\end{proof}

\begin{proposition}\label{prop3}
We have $\mathcal{Z}(\R_{ij}) \subseteq \R_{ij} + \mathcal{Z}(\R)$ with $i \neq j$.
\end{proposition}
\begin{proof}
We will make just the case $i=1$, $j=2$ because the other case it is similar.
For any $a \in \mathcal{Z}(\R_{12})$ with $a = a_{11} + a_{12} + a_{21} + a_{22}$ we have 
$[a,r_{12}] = 0$ which implies $[a_{11}+a_{22},r_{12}] = 0$ and $a_{21}r_{12} = 0$ for all $r_{12} \in \R_{12}$.
By item $(\spadesuit)$ of the Proposition \ref{prop2} and item $(1)$ follows that $a_{11} + a_{22} \in \mathcal{Z}(\R)$ and $a_{21} = 0$.
Therefore $a = a_{11} + a_{12} + a_{21} + a_{22} = a_{12} + a_{11} + a_{22} \in \R_{12} + \mathcal{Z}(\R)$
\end{proof}

There are several results on the characterizations of Lie derivations on
associative rings. The first characterization on Lie derivations is due to Martindale III, see \cite{Mart}, who proved the following result in $1964$.

\begin{theorem}(Martindale III)
Let $L$ be a Lie derivation of a primitive ring $R$ into itself,
where $R$ contains a nontrivial idempotent and the characteristic of $R$ is not $2$ then
every Lie derivation $L$ of $R$ is of the form $L = D + T$, where $D$ is an ordinary
derivation of $R$ into a primitive ring $\bar{R}$
containing $R$ and $T$ is an additive mapping
of $R$ into the center of $\bar{R}$
that maps commutators into zero.
\end{theorem}

In $1993$, Bre$\check{s}$ar generalized the above characterization of Lie derivations on primitive rings to those on prime rings, see \cite{bresar}. He obtained the following theorem for associative rings.

\begin{theorem}(Bre$\check{s}$ar) Let $R$ be a prime ring of characteristic not $2$. Let $d$ be a Lie
derivation of $R$. If $R$ does not satisfy $S_4$, then $d$ is of the form $\delta + \tau$, where $\delta$ is a
derivation of $R$ into its central closure and $\tau$ is an additive mapping of $R$ into its
extended centroid sending commutators to zero.
\end{theorem}

As for characterizations on Lie derivable mappings on operator algebras, the
following result is proved in \cite{lu}.
\begin{theorem}
Let $X$ be a Banach space of dimension greater than $1$ and $\delta$ be a
Lie derivable mapping of $B(X)$ into itself. Then $\delta = D+\tau$ , where $D$ is an additive
derivation and $\tau$ is a map from $B(X)$ into $\mathbb{F}I$ vanishing at commutators.
\end{theorem}

In view of the above, this motivated us to study the characterization of Lie multiplicative derivation on alternative rings.

\section{Main theorem}

We shall prove as follows the main result of this paper.

\begin{theorem}\label{mainthm} 
Let $\R$ be an unital alternative ring with nontrivial idempotent and $\D \colon  \R \rightarrow \R$ is a Lie multiplicative derivation. Then
$\D$ is the form $\delta + \tau$, where $\delta$ is an additive derivation of $\R$ and $\tau$ is a map from $\R$ into its center $\mathcal{Z}(\R)$, which maps
commutators into the zero if and only if
\begin{enumerate}
\item[\it a)] $e_2\D(\R_{11})e_2 \subseteq \mathcal{Z}(\R) e_2,$
\item[\it b)] $e_1\D(\R_{22})e_1 \subseteq \mathcal{Z}(\R) e_1.$
\end{enumerate}
\end{theorem}

Firstly let us assume that Lie multiplicative derivation $\D \colon  \R \rightarrow \R$ satisfies $a)$ and $b)$.
The following Lemmas has the same hypotheses of Theorem \ref{mainthm} and we need these Lemmas for the proof of the first part this Theorem. Thus, let us consider $e_{1}$ a nontrivial idempotent of $\mathfrak{R}$. We started with the following

\begin{lemma}\label{lema1} $\D(0) = 0$.
\end{lemma}
\begin{proof}
In deed, $\D(0) = \D([0,0]) = [\D(0),0] + [0, \D(0)] = 0$. 
\end{proof} 

\begin{lemma}\label{lema2} $\D(e_1) - f_{y,z}(e_1) \in \mathcal{Z}(\R)$, with $y= \D(e_1)_{12} + \D(e_1)_{21}$, $z = -e_1$ where $f_{y,z} := [L_y, L_z] + [L_y, R_z] + [R_y, R_z]$ and $L$, $R$ are left and right multiplication operators respectively.
\end{lemma}
\begin{proof}
Firstly observe that
\begin{eqnarray*}
\D(a_{12}) &=& \D([e_1, a_{12}]) = [\D(e_1), a_{12}] + [e_1, \D(a_{12})] \\&=& \D(e_1)a_{12} - a_{12}\D(e_1) + e_1\D(a_{12}) - \D(a_{12})e_1
\end{eqnarray*}
In the above equation the left and right sides are, respectively, multiplied by $e_1$ and $e_2$, and we have $[\D(e_1)_{11} + \D(e_1)_{22} , a_{12}] = 0$ for all $a_{12} \in \R_{12}$ so by $(\spadesuit)$ of Proposition \ref{prop2} we get $\D(e_1)_{11} + \D(e_1)_{22} \in \mathcal{Z}(\R)$.
Taking $y= \D(e_1)_{12} + \D(e_1)_{21}$ and $z= -e_1$ we obtain $\D(e_1) - f_{y,z}(e_1) = \D(e_1)_{11} + \D(e_1)_{22} \in \mathcal{Z}(\R)$.
\end{proof}

Before we continue it is worth noting that $f_{y,z} := [L_y, L_z] + [L_y, R_z] + [R_y, R_z]$ is a derivation, by [$77$ page of \cite{sch}], so without loss of generality, we can assume that
$\D(e_1) \in \mathcal{Z}(\R)$

\begin{remark}\label{obs1}
If $\D(e_1) \in \mathcal{Z}(\R)$ then $\D(e_2) \in \mathcal{Z}(\R)$. Indeed, note that $[e_1, a]=[a, e_2]$ for all $a \in \R$. Thus 
\begin{eqnarray*}
[\D(e_2), a] &=& \D([e_2, a]) - [e_2, \D(a)] = \D([a,e_1]) - [e_2, \D(a)] \\&=& [\D(a), e_1] + [a, \D(e_1)] - [e_2, \D(a)] \\&=& [\D(a), e_1] - [e_2, \D(a)] \\&=& [\D(a), e_1] + [e_1, \D(a)] = 0
\end{eqnarray*}
for all $a \in \R$. Therefore $\D(e_2) \in \mathcal{Z}(\R)$. Moreover, $\D(a_{ij}) \in \R_{ij}$ with $i \neq j$, because $\D(a_{ij}) = \D([e_i, a_{ij}]) = [e_i, \D(a_{ij})]$.
\end{remark}

\begin{lemma}\label{lema3} $\D(\R_{ii}) \subseteq \R_{ii} + \mathcal{Z}(\R) \ (i = 1,2)$
\end{lemma}
\begin{proof}
We show just the case $i = 1$ because the other case can be treated similarly. For every $a_{11} \in \R_{11}$, with $\D(a_{11}) = b_{11} + b_{12} + b_{21} + b_{22}$ we get
$$0 = \D([a_{11}, e_1]) = [\D(a_{11}),e_1] + [a_{11}, \D(e_1)] = [\D(a_{11}),e_1].$$
From this $b_{12} = b_{21} = 0$. By Theorem \ref{mainthm} item $a)$, we have
$$\D(a_{11}) = b_{11} + e_2\D(a_{11})e_2 = b_{11} + ze_2 = b_{11} -e_1z + z \in \R_{11} + \mathcal{Z}(\R).$$
\end{proof}

\begin{lemma}\label{lema4}
$\D$ is almost additive map, that is, for every $a,b \in \R$, $\D(a+b) - \D(a) - \D(b) \in \mathcal{Z}(\R)$.
\end{lemma}
\begin{proof}
For proof to see Theorem $4.1$ in \cite{posd}. 
\end{proof}

Now let us define the mappings $\delta$ and $\tau$. By Remark \ref{obs1} and Lemma \ref{lema3} we have that

\begin{enumerate}
\item[\it (A)] if $a_{ij} \in \R_{ij}$, $i \neq j$, then $\D(a_{ij}) = b_{ij} \in \R_{ij}$,
\item[\it (B)] if $a_{ii} \in \R_{ii}$, then $\D(a_{ii}) = b_{ii} + z, b_{ii} \in \R_{ii}$, $z$ is a central element.
 \end{enumerate} 
We note that in $(B)$, $b_{ii}$ and $z$ are uniquely determined. Indeed, if $\D(a_{ii}) = b'_{ii}+ z'$, $b'_{ii} \in \R_{ii}$, $z' \in \mathcal{Z}(\R)$. Then $b_{ii} − b'_{ii} \in \mathcal{Z}(\R)$. Hence by conditions $(2)$ and $(3)$, $b_{ii} = b'_{ii}$ and $z = z'$.
Now we define a map $\delta$ of $\R$ according to the rule $\delta(a_{ij}) = b_{ij}, a_{ij} \in \R_{ij}$. For every $a = a_{11} + a_{12} + a_{21} + a_{22} \in \R$, define $\delta(a) = \sum \delta(a_{ij})$. And a map $\tau$ of $\R$ into $\mathcal{Z}(\R)$ is then defined by
\begin{eqnarray*}
\tau(a) &=& \D(a) - \delta(a) \\&=& \D(a) - (\delta(a_{11}) + \delta(a_{12}) + \delta(a_{21}) + \delta(a_{22})) \\&=& \D(a) - (b_{11} + b_{12} + b_{21} + b_{22}) \\&=& \D(a) - (\D(a_{11}) - z_{a_{11}} + \D(a_{12}) + \D(a_{21}) + \D(a_{22}) - z_{a_{22}}) \\&=& \D(a) - (\D(a_{11}) + \D(a_{12}) + \D(a_{21}) + \D(a_{22}) - (z_{a_{11}} + z_{a_{22}})) \\&=& \D(a) - (\D(a_{11}) + \D(a_{12}) + \D(a_{21}) + \D(a_{22})). 
\end{eqnarray*}
We need to prove that $\delta$ and $\tau$ are desired maps. 

\begin{lemma}\label{lema5}
$\delta$ is an additive map.
\end{lemma}
\begin{proof}
We only need to show that $\delta$ is an additive on $\R_{ii}$. Let $a_{ii}, b_{ii} \in \R_{ii}$,
\begin{eqnarray*}
\delta(a_{ii} + b_{ii}) - \delta(a_{ii}) - \delta(b_{ii}) &=& \D(a_{ii} + b_{ii}) - \tau(a_{ii} + b_{ii}) - \D(a_{ii}) \\&+& \tau(a_{ii}) - \D(b_{ii}) + \tau(b_{ii}).
\end{eqnarray*} 
Thus, $\delta(a_{ii} + b_{ii}) - \delta(a_{ii}) - \delta(b_{ii}) \in \mathcal{Z}(\R) \cap \R_{ii} = \left\{0\right\}$.
\end{proof}
Now we show that $\delta(ab) = \delta(a)b + a\delta(b)$ for all $a, b \in \R$.

\begin{lemma}\label{lema6}
For every $a_{ii}, b_{ii} \in \R_{ii}$, $a_{ij}, b_{ij} \in \R_{ij}$, $b_{ji} \in \R_{ji}$ and $b_{jj} \in \R_{jj}$ with $i \neq j$ we have
\begin{enumerate}
\item[\it (I)] $\delta(a_{ii}b_{ij}) = \delta(a_{ii})b_{ij} + a_{ii}\delta(b_{ij})$,
\item[\it (II)] $\delta(a_{ij}b_{jj}) = \delta(a_{ij})b_{jj} + a_{ij}\delta(b_{jj})$,
\item[\it (III)] $\delta(a_{ii}b_{ii}) = \delta(a_{ii})b_{ii} + a_{ii}\delta(b_{ii})$,
\item[\it (IV)] $\delta(a_{ij}b_{ij}) = \delta(a_{ij})b_{ij} + a_{ij}\delta(b_{ij})$,
\item[\it (V)] $\delta(a_{ij}b_{ji}) = \delta(a_{ij})b_{ji} + a_{ij}\delta(b_{ji}).$
\end{enumerate}
\end{lemma}
\begin{proof}
Let us start with $(I)$
\begin{eqnarray*}
\delta(a_{ii}b_{ij}) &=& \D(a_{ii}b_{ij})= \D([a_{ii}, b_{ij}]) \\&=& [\D(a_{ii}), b_{ij}] + [a_{ii}, \D(b_{ij})] \\&=& [\delta(a_{ii}), b_{ij}] + [a_{ii}, \delta(b_{ij})] \\&=& \delta(a_{ii})b_{ij} + a_{ii}\delta(b_{ij}).
\end{eqnarray*} 
Next $(II)$
\begin{eqnarray*}
\delta(a_{ij}b_{jj}) &=& \D(a_{ij}b_{jj})= \D([a_{ij}, b_{jj}]) \\&=& [\D(a_{ij}), b_{jj}] + [a_{ij}, \D(b_{jj})] \\&=& [\delta(a_{ij}), b_{jj}] + [a_{ij}, \delta(b_{jj})] \\&=& \delta(a_{ij})b_{jj} + a_{ij}\delta(b_{jj}).
\end{eqnarray*} 
Now we show $(III)$. By Proposition \ref{Lflexivel} and $(I)$ we get
$$\delta((a_{ii}b_{ii})r_{ij}) = \delta(a_{ii}b_{ii})r_{ij} + (a_{ii}b_{ii})\delta(r_{ij}).$$ 
On the other hand,
\begin{eqnarray*}
\delta(a_{ii}(b_{ii}r_{ij})) &=& \delta(a_{ii})b_{ii}r_{ij} + a_{ii}\delta(b_{ii}r_{ij}) \\&=& \delta(a_{ii})b_{ii}r_{ij} + a_{ii}(\delta(b_{ii})r_{ij} + b_{ii}\delta(r_{ij})).
\end{eqnarray*}
As $(a_{ii}b_{ii})r_{ij} = a_{ii}(b_{ii}r_{ij})$ and $(a_{ii}b_{ii})\delta(r_{ij}) = a_{ii}(b_{ii}\delta(r_{ij}))$ we obtain
$$(\delta(a_{ii}b_{ii}) - \delta(a_{ii})b_{ii} - a_{ii}\delta(b_{ii}))r_{ij} = 0$$
for all $r_{ij} \in \R_{ij}$.
So $\delta(a_{ii}b_{ii}) = \delta(a_{ii})b_{ii} + a_{ii}\delta(b_{ii})$.

Next ($IV$). 
\begin{eqnarray*}
2\delta(a_{ij}b_{ij}) &=& \delta(2a_{ij}b_{ij}) = \D(2a_{ij}b_{ij}) \\&=& \D([a_{ij},b_{ij}]) = [\D(a_{ij}), b_{ij}] + [a_{ij}, \D(b_{ij})] \\&=& [\delta(a_{ij}), b_{ij}] + [a_{ij}, \delta(b_{ij})] \\&=& \delta(a_{ij})b_{ij} - b_{ij}\delta(a_{ij}) + a_{ij}\delta(b_{ij}) - \delta(b_{ij})a_{ij} \\&=& 2(\delta(a_{ij})b_{ij} + b_{ij}\delta(a_{ij})) 
\end{eqnarray*}
As $\R$ is $2$- torsion free it's follow that $\delta(a_{ij}b_{ij}) = \delta(a_{ij})b_{ij} + a_{ij}\delta(b_{ij})$. 
And finally we show $(V)$. We have
\begin{eqnarray*}
\tau([a_{ij}, b_{ji}]) &=& \D([a_{ij}, b_{ji}]) - \delta([a_{ij}, b_{ji}]) \\&=& [\D(a_{ij}), b_{ji}] + [a_{ij}, \D(b_{ji})] - \delta(a_{ij}b_{ji} - b_{ji}a_{ij}) \\&=& [\delta(a_{ij}), b_{ji}] + [a_{ij}, \delta(b_{ji})] - \delta(a_{ij} b_{ji}) + \delta(b_{ji}a_{ij}) \\&=& 
\delta(a_{ij})b_{ji} - b_{ji}\delta(a_{ij}) + a_{ij}\delta(b_{ji}) - \delta(b_{ji})a_{ij} - \delta(a_{ij}b_{ji}) \\&+& \delta(b_{ji}a_{ij}), 
\end{eqnarray*}
which implies 
$$[\delta(a_{ij})b_{ji} + a_{ij}\delta(b_{ji}) - \delta(a_{ij}b_{ji})] + [\delta(b_{ji}a_{ij}) - \delta(b_{ji})a_{ij} - b_{ji}\delta(a_{ij})] = z \in \mathcal{Z}(\R).$$
If $z= 0$ then $\delta(a_{ij}b_{ji}) = \delta(a_{ij})b_{ji} + a_{ij}\delta(b_{ji}).$
If $z \neq 0$ we multiply by $a_{ij}$ we get
$$a_{ij}\delta(b_{ji}a_{ij}) - a_{ij}\delta(b_{ji})a_{ij} - a_{ij}(b_{ji}\delta(a_{ij})) = a_{ij}z.$$
By $(II)$ we have
\begin{eqnarray}\label{dif}
\delta(a_{ij}b_{ji}a_{ij}) - \delta(a_{ij})(b_{ji}a_{ij})- a_{ij}\delta(b_{ji})a_{ij} - a_{ij}(b_{ji}\delta(a_{ij})) = a_{ij}z.
\end{eqnarray}
Now we observe that $\delta(a_{ij}b_{ji}a_{ij}) = \delta(a_{ij})(b_{ji}a_{ij})+ a_{ij}\delta(b_{ji})a_{ij} + a_{ij}(b_{ji}\delta(a_{ij}))$. In deed, observe that $[[a_{ij}, b_{ji}],a_{ij}] = 2a_{ij}b_{ji}a_{ij}$. Then
\begin{eqnarray*}
2\delta(a_{ij}b_{ji}a_{ij}) &=& \delta(2a_{ij}b_{ji}a_{ij}) \\&=& \D([[a_{ij}, b_{ji}],a_{ij}]) \\&=& [[\D(a_{ij}), b_{ji}],a_{ij}] + [[a_{ij}, \D(b_{ji})],a_{ij}] + [[a_{ij}, b_{ji}],\D(a_{ij})] \\&=& [[\delta(a_{ij}), b_{ji}],a_{ij}] + [[a_{ij}, \delta(b_{ji})],a_{ij}] + [[a_{ij}, b_{ji}],\delta(a_{ij})] \\&=& (\delta(a_{ij})b_{ji})a_{ij} + a_{ij}(b_{ji}\delta(a_{ij})) + 2a_{ij}\delta(b_{ji})a_{ij} \\&+& (a_{ij}b_{ji})\delta(a_{ij}) + \delta(a_{ij})(b_{ji}a_{ij}) \\&=& \delta(a_{ij})(b_{ji}a_{ij}) - (a_{ij}b_{ji})\delta(a_{ij}) + a_{ij}(b_{ji}\delta(a_{ij})) \\&+& a_{ij}(b_{ji}\delta(a_{ij})) + 2a_{ij}\delta(b_{ji})a_{ij} + (a_{ij}b_{ji})\delta(a_{ij}) + \delta(a_{ij})(b_{ji}a_{ij}) \\&=& 2(\delta(a_{ij})(b_{ji}a_{ij})+ a_{ij}\delta(b_{ji})a_{ij} + a_{ij}(b_{ji}\delta(a_{ij}))).  
\end{eqnarray*}
Since $\R$ is $2$-torsion free we get $\delta(a_{ij}b_{ji}a_{ij}) = \delta(a_{ij})(b_{ji}a_{ij})+ a_{ij}\delta(b_{ji})a_{ij} + a_{ij}(b_{ji}\delta(a_{ij}))$.
So $a_{ij}z = 0$ but by $(4)$ we have $zh = e_1 + e_2$ and $a_{ij} = 0$ which is a contradiction. Therefore $\delta(a_{ij}b_{ji}) = \delta(a_{ij})b_{ji} + a_{ij}\delta(b_{ji}).$
   
\end{proof}

\begin{lemma}\label{lema7}
$\delta$ is a derivation.
\end{lemma}
\begin{proof}
Let be $a, b \in \R$. We have
\begin{eqnarray*}
\delta(ab) &=& \delta((a_{11} + a_{12} + a_{21} + a_{22})(b_{11} + b_{12} + b_{21} + b_{22})) \\&=& \delta(a_{11}b_{11}) + \delta(a_{11}b_{12}) + \delta(a_{12}b_{12}) + \delta(a_{12}b_{21}) + \delta(a_{12}b_{22}) \\&+& \delta(a_{21}b_{11}) + \delta(a_{21}b_{12}) + \delta(a_{21}b_{21}) + \delta(a_{22}b_{21}) + \delta(a_{22}b_{22}) \\&=& 
\delta(a)b + a \delta(b)
\end{eqnarray*}
by Lemmas \ref{lema5} and \ref{lema6}.
\end{proof}

\begin{lemma}\label{lema8}
$\tau$ sends the commutators into zero.
\end{lemma}
\begin{proof}
\begin{eqnarray*}
\tau([a, b]) &=& \D([a, b]) - \delta([a, b])
\\&=& [\D(a) , b] + [a,\D(b)] - \delta([a, b])
\\&=& [\delta(a) , b] + [a, \delta(b)] - \delta([a, b])
\\&=& 0.
\end{eqnarray*}
\end{proof}

Let us assume that $\D\colon  \R \rightarrow \R$ is a Lie multiplicative derivation of the form $\D = \delta + \tau$ where $\delta$ is an additive derivation of $\R$ and $\tau$ is a map from $\R$ into its center $\mathcal{Z}(\R)$, which maps commutators into the zero.
So
\begin{eqnarray*}
e_{2}\D(a_{11})e_2 &=& e_2\delta(a_{11})e_2 + e_2\tau(a_{11})e_2 \\&=& e_2\delta(e_1 a_{11})e_2 + e_2\tau(a_{11})e_2\\&=& e_2(\delta(e_1)a_{11} +e_1 \delta(a_{11}))e_2 + e_2\tau(a_{11})e_2 \\&=& e_2(\delta(e_1)a_{11})e_2 + e_2(e_1 \delta(a_{11}))e_2 + e_2\tau(a_{11})e_2 \\&=&(e_2\delta(e_1))(a_{11}e_2) + (e_2e_1) (\delta(a_{11})e_2) + e_2\tau(a_{11})e_2 \\&=& e_2\tau(a_{11})e_2 \in \mathcal{Z}(\R)e_2.
\end{eqnarray*}

and

\begin{eqnarray*}
e_{1}\D(a_{22})e_1 &=& e_1\delta(a_{22})e_1 + e_1\tau(a_{22})e_1 \\&=& e_1\delta(e_2 a_{22})e_1 + e_1\tau(a_{22})e_1\\&=& e_1(\delta(e_2)a_{22} +e_2 \delta(a_{22}))e_1 + e_1\tau(a_{22})e_1 \\&=& e_1(\delta(e_2)a_{22})e_1 + e_1(e_2 \delta(a_{22}))e_1 + e_1\tau(a_{22})e_1 \\&=&(e_1\delta(e_2))(a_{22}e_1) + (e_1e_2) (\delta(a_{22})e_1) + e_1\tau(a_{22})e_1 \\&=& e_1\tau(a_{22})e_1 \in \mathcal{Z}(\R)e_1,
\end{eqnarray*}
for every $a_{11} \in \R_{11}$ and $a_{22} \in \R_{22}$. This demonstrates the letters $a)$ and $b)$ and the proof of the Theorem \ref{mainthm} is complete.

\vspace{0,5cm}

\section{Applications}

\begin{corollary}
Let $\R$ be an unital prime alternative ring with nontrivial idempotent satisfying $(4)$ and $\D \colon  \R \rightarrow \R$ is a Lie multiplicative derivation. Then
$\D$ is the form $\delta + \tau$, where $\delta$ is an additive derivation of $\R$ and $\tau$ is a map from $\R$ into its center $\mathcal{Z}(\R)$, which maps
commutators into the zero if and only if
\begin{enumerate}
\item[\it a)] $e_2\D(\R_{11})e_2 \subseteq \mathcal{Z}(\R) e_2,$
\item[\it b)] $e_1\D(\R_{22})e_1 \subseteq \mathcal{Z}(\R) e_1.$
\end{enumerate}
\end{corollary}

And we finished the article with an application on simple alternative rings.

\begin{corollary}
Let $\R$ be an unital simple alternative ring with nontrivial idempotent and $\D \colon  \R \rightarrow \R$ is a Lie multiplicative derivation. Then
$\D$ is the form $\delta + \tau$, where $\delta$ is an additive derivation of $\R$ and $\tau$ is a map from $\R$ into its center $\mathcal{Z}(\R)$, which maps
commutators into the zero if and only if
\begin{enumerate}
\item[\it a)] $e_2\D(\R_{11})e_2 \subseteq \mathcal{Z}(\R) e_2,$
\item[\it b)] $e_1\D(\R_{22})e_1 \subseteq \mathcal{Z}(\R) e_1.$
\end{enumerate}
\end{corollary}
\begin{proof}
It is enough to observe that every single ring is prime and $\mathcal{Z}(\R)$ is a field.
\end{proof}

\end{document}